\documentclass[12pt, reqno]{amsart}
\usepackage{amsmath, amsthm, amscd, amsfonts, amssymb, graphicx, color}
\usepackage[bookmarksnumbered, colorlinks, plainpages]{hyperref}
\hypersetup{colorlinks=true,linkcolor=red, anchorcolor=green, citecolor=cyan, urlcolor=red, filecolor=magenta, pdftoolbar=true}

\newtheorem{theorem}{Theorem}[section]
\newtheorem{lemma}[theorem]{Lemma}
\newtheorem{proposition}[theorem]{Proposition}
\newtheorem{corollary}[theorem]{Corollary}
\theoremstyle{definition}

\theoremstyle{remark}

\numberwithin{equation}{section}

\newcommand{\PP}{\mathcal{P}}

\newcommand{\R}{\mathbb{R}}

\newcommand{\HH}{\mathbb{H}}
\newcommand{\MM}{\mathbb{M}}

\newcommand{\A}{B(\HH)}

\newcommand{\leqs}{\leqslant}
\newcommand{\geqs}{\geqslant}

\newcommand{\ap}{\alpha}
\newcommand{\bt}{\beta}

\newcommand{\ep}{\epsilon}
\newcommand{\ld }{\lambda}

\newcommand{\Om}{\Omega}

\newcommand{\sm}{\,\sigma\,}

\newcommand{\norm}[1]{\lVert #1 \rVert}

\newcommand{\Kant}{\dfrac{a^2+b^2}{2ab}}
\newcommand{\IntAweight}[1]{\int_{\Om} W_t^{\frac{1}{2}} #1 W_t^{\frac{1}{2}} \,d \mu(t)}

\newcommand{\IntWeightT}[1]{\int_{\Om} {W_{t}}^{\frac{1}{2}} {#1} {W_{t}}^{\frac{1}{2}} \,d \mu(t)}

\DeclareMathOperator{\Sp}{Sp}

\begin{document}
\setcounter{page}{1}

\title[Kantorovich Type Integral Inequalities]{Kantorovich Type Integral Inequalities for Tensor Product of Continuous Fields of Hilbert Space Operators}

\author[P. Chansangiam]{Pattrawut Chansangiam$^{*}$}

\address{$^{*}$ Department of Mathematics, Faculty of Science, King Mongkut's Institute of Technology
Ladkrabang, Chalongkrung Rd., Ladkrabang, Bangkok 10520, Thailand.}
\email{\textcolor[rgb]{0.00,0.00,0.84}{pattrawut.c@gmail.com}}



\date{Received: xxxxxx; Revised: yyyyyy; Accepted: zzzzzz.
\newline \indent $^{*}$ Corresponding author}

\begin{abstract}
This paper presents a number of Kantorovich type integral inequalities involving
	tensor products of continuous fields of bounded linear operators on a Hilbert space. 
	Kantorovich type inequality in which the product is replaced by an operator mean is also considered.
	Such inequalities include discrete inequalities as special cases.
	Moreover, some generalizations of an additive Gr\"{u}ss integral inequality for operators
	%reverse versions of weighted arithmetic-geometric mean 
	%inequality for multiple operators 
	are obtained.
\end{abstract} 

\maketitle

\section{Introduction}

The classical Kantorovich inequality \cite{Kantorovich} asserts that for real numbers $a_i$ and $w_i$
such that $0<a\leqs a_i \leqs b$ and $w_i \geqs 0$ for all $1 \leqs i \leqs n$, we have

\begin{equation} \label{eq: classical Kant ineq}
	\left( \sum_{i=1}^n w_i a_i \right) \left( \sum_{i=1}^n \frac{w_i}{a_i} \right)
	\:\leqs\: \frac{(a+b)^2}{4ab} \left( \sum_{i=1}^n w_i \right)^2.
\end{equation}
This inequality can be viewed as a reverse weighted arithmetic-harmonic mean (AM-HM) inequality.
The bound \eqref{eq: classical Kant ineq} is used for convergence analysis in 
numerical methods and statistics.
Over the years, various generalizations, variations and refinements of this inequality in several settings
have been investigated by many authors. % (e.g. \cite{Baksalary, Marshall-Olkin, Kitamura-Seo}).
This inequality has been proved to be equivalent to many inequalities, e.g. 
Cauchy-Schwarz-Bunyakovsky inequality and Wielant's inequality; see also  \cite{Greub-Rheinboldt, Zhang}.
%Greub and Rheinboldt {Greub-Rheinboldt} provided equivalent inequalities to Kantorovich inequality.
In the literature, there is an integral version of Kantorovich inequality as follows. 
For  a Riemann integrable function $f:[\ap,\bt] \to \R$
with $a \leqs f(x) \leqs b$ for all $x \in [\ap,\bt]$, we have (e.g. \cite{Mitrinovic})
\begin{align}
	\int_{\ap}^{\bt} f(x)^2 \,dx \leqs \frac{(a+b)^2}{4ab} \left( \int_{\ap}^{\bt}  f(x)\,dx \right)^2.
	\label{eq: additive Gruss ineq}
\end{align}
This inequality is also called an additive version of Gr\"{u}ss inequality.
.

%\subsection{Matrix Kantorovich inequalities}

Many matrix versions of Kantorovich inequality were obtained in the literature,
e.g. \cite{Baksalary, Liu-Neudecker, Marshall-Olkin}.
Let $\MM_k$ be the algebra of $k$-by-$k$ complex matrices.
Recall the Hadamard product of $A, B \in \MM_k$ is defined to be the entrywise product:
\begin{align*}
	A \circ B \:=\: [a_{ij} \,b_{ij}] \:\in\: \MM_k.
\end{align*}
A matrix analogue of this inequality involving Hadamard product 
is given in \cite{Matharu-Aujla} as follows.

\begin{theorem}[\cite{Matharu-Aujla}, Theorem 2.2]   \label{thm: Kant Hadamard - Matrix}
	For each $i=1,2,\dots,n$, let $A_i \in \MM_k$ 
	be a positive definite matrix such that $0<aI \leqs A_i \leqs bI$
	and $W_i \in \MM_k$ a positive semidefinite matrix. Then
	\begin{equation}
		\sum_{i=1}^n W_i^{\frac{1}{2}} A_i W_i^{\frac{1}{2}} \circ \sum_{i=1}^n W_i^{\frac{1}{2}} A_i^{-1} W_i^{\frac{1}{2}}
		\leqs \Kant \left(\sum_{i=1}^n W_i \circ \sum_{i=1}^n W_i \right).
	\end{equation}
\end{theorem}

Kantorovich inequality can be regarded as a reverse of the following Fiedler's inequality:
\begin{align*}
		A \circ A^{-1} \:\geqs\: I.
	\end{align*}
for any positive definite matrix $A \in \MM_k$.

Operator versions of Kantorovich inequality was investigated, for instance, 
in \cite{Dragomir, Dragomir2012, Furuta, Moslehian} and references therein.
Kantorovich type inequality where the product is replaced by an operator mean
was considered in \cite{Nakamoto-Nakamura, Yamazaki}.

In this paper, we establish various integral inequalities of Kantorovich type
for continuous field of Hilbert space operators.
The inequalities \eqref{eq: classical Kant ineq} and \eqref{eq: additive Gruss ineq}  
are generalized in many ways in terms of 
Bochner integrals on the Banach space of bounded linear operators.
The products between two operators considered here are the Hilbert tensor product.
Moreover, Kantorovich type inequalities involving Kubo-Ando operator means are obtained.
Such integral inequalities include discrete inequalities as special cases.
In particular, we get some generalizations of additive Gr\"{u}ss type inequality for operators.

This paper is organised as follows.
We set up basic notations  about continuous fields of operators 
and state the main assumption used throughout the paper in Section 2.
Then Section 3 deals with Kantorovich type integral inequalities 
involving tensor product of continuous fields of operators.
In Section 4,  we first recall Kubo-Ando theory of operator means and then derive Kantorovich integral inequalities involving operator means.
In the last section, we derive further operator integral inequalities, including 
additive Gr\"{u}ss inequality.

\section{Continuous field of operators and its integralability}

\subsection{Continuous field of operators}
Throughout this paper, let $\HH$ be a complex Hilbert space. %$\CC$.
Denote by $B(\HH)$ the C$^*$-algebra of bounded linear operators acting on $\HH$.
The spectrum of $A \in B(\HH)$ is written as $\Sp(A)$.
We shall write $I$ for the identity operator on a Hilbert space; 
the space mentioned here should be clear from the context.
 
%Let $\A$ be a unital C$^*$-algebra of $B(\HH)$ and $\B$ a unital C$^*$-subalgebra of $\A$. 
%Let $\Omega$ be a compact Hausdorff space endowed with a Radon measure $\mu$.

Let $\Omega$ be a locally compact Hausdorff space endowed with a finite Radon measure $\mu$.
A  family $(A_t)_{t \in \Omega}$ of operators in $\A$ is said to be 
a \emph{continuous field of operators} if the parametrization $t \mapsto A_t$ is norm continuous
on $\Omega$.  
If, in addition, the norm function $t \mapsto \norm{A_t}$ is Lebesgue integrable on $\Omega$,
then we can form the Bochner integral of $A_t's$ as follows (see also \cite{Hansen-Pecaric-Peric}).
Let $\PP$ be a partition of $\Om$ into disjoint Borel subsets and let $\ep >0$ be a real number.
For each operator
$A_t$ in $\A$, we can approximate $A_t$ by a net of operators in the form
\begin{align*}
	F_{\PP, \ep} (A_t) = \sum_{i=1}^n \mu(E_i) A_{t_i} 
\end{align*} 
where $E_i \in \PP$ and
$
	t_i \in E_i \subseteq \{t \in \Om \,:\, \norm{A_t - A_{t_i}} < \epsilon\} 
$ 
for each $1 \leqs i \leqs n$.
Then the net $F_{\PP,\ep}(A_t)$ converges uniformly to the Bochner integral
\begin{align*}
	\int_{\Om} A_t \,d\mu(t).
\end{align*}
The set of continuous functions from $\Om$ to $\A$
becomes a C$^*$-algebra under the pointwise operations and the C$^*$-norm
\begin{align*}
	\norm{(A_t)_{t \in \Om}} = \sup_{t \in \Om} \,\norm{A_t}.
\end{align*}

%Now, we state the main assumption for most of the results in this paper.

\subsection{Main hypothesis}
%Let $\A$ be a unital C$^*$-algebra of $B(\HH)$ and $\B$ a unital C$^*$-subalgebra of $\A$. 
%Let $\Omega$ be a locally compact Hausdorff space equipped with a finite Radon measure $\mu$.
Let $(A_t)_{t \in \Om}$  be a bounded continuous field of strictly positive operators
in $\A$ such that 
\begin{itemize}
	\item[-]	the norm function $t \mapsto \norm{A_t}$ is Lebesgue integrable on $\Omega$ 
	\item[-]	$\Sp(A_t) \subseteq [a,b] \subseteq (0,\infty)$ for each $t \in \Om$.
\end{itemize}
Let $(W_t)_{t \in \Om}$ be a continuous field of positive operators in $\A$.

\begin{proposition} \label{prop: Bochner integralability}
	Assume Main hypothesis. For any continuous function $f:[a,b] \to \R$, 
	we can form the Bochner integral
\begin{align*}
	\int_{\Om} W_t^{\frac{1}{2}} f(A_t) W_t^{\frac{1}{2}} \,d\mu(t).
\end{align*}
In addition, if $f([a,b]) \subseteq [0,\infty)$, then this operator is positive.
\end{proposition}
\begin{proof}
	Since $(\Om, \mu)$ is a finite measure space, it suffices to prove the Lebesgue integralability
	of its norm function. Indeed, we have
		\begin{align*}
		\int_{\Om} \norm{W_t^{\frac{1}{2}} f(A_t) W_t^{\frac{1}{2}}} \,d\mu(t) 
		&\leqs \int_{\Om} \norm{W_t^{\frac{1}{2}}} \cdot \norm{f(A_t)} \cdot \norm{W_t^{\frac{1}{2}}} \,d\mu(t) \\
		&\leqs \int_{\Om} \norm{W_t} \cdot \norm{f}_{\infty}   \,d\mu(t) \\
		&\leqs \int_{\Om} \sup_{t \in \Om} \norm{W_t} \cdot \norm{f}_{\infty}  \,d\mu(t) \\
		&= \mu(\Om)  \norm{f}_{\infty} \sup_{t \in \Om} \norm{W_t}    \\
		&< \infty.
	\end{align*}
	Suppose that $f$ is positive on $[a,b]$. 
	Then $f(A_t)$ is a positive operator for each $t \in \Om$.
	It follows that the integral is also positive. 
	%which implies the Bochner integralability of the operator-valued function 
	%$t \mapsto  W_t^{\frac{1}{2}} f(A_t) W_t^{\frac{1}{2}}$ on $\Om$.
	%For the latter case, use the above process and the fact that $(1/M)I \leqs A_t^{-1} \leqs (1/m)I$ which implies
	%$\norm{A_t^{-1}} \leqs 1/m$.
\end{proof}

\section{Integral inequalities of Kantorovich type for tensor product of operators}

In this section, we derive many integral inequalities of Kantorovich type for operators
in which the product is given by the tensor product.
Such inequalities includes discrete inequalities as special cases.
In particular, we get a reverse of weighted AM-HM operator inequality.

\subsection{Tensor products}
For each fixed $X \in B(\HH)$, the map $A \mapsto A \otimes X$ and the map $A \mapsto X \otimes A$
are bounded linear operators
from $B(\HH)$ to $B(\HH \otimes \HH)$.
It follows that
\begin{equation}
	\int_{\Om} A_t \,d \mu(t) \otimes X \:=\: \int_{\Om}( A_t \otimes X) \,d \mu(t). 
	\label{eq: int and tensor product}
\end{equation}
Moreover, these maps preserve positivity when the multiplier is a positive operator.
For each $A,B \in B(\HH)$, we denote
\begin{align*}
	A \otimes_s B \:=\: \frac{1}{2}(A \otimes B + B \otimes A).
\end{align*}
Recall that the tensor power $A^{\otimes 2}$ is defined to be $A \otimes A$.

We start with the following estimation about tensor products. 

\begin{lemma} \label{lem: Lemma 1}
	The minimum constant $k$ for which the inequality 
	\begin{align}
		A \otimes B^{-1} + A^{-1} \otimes B \:\leqs\:  k I.  
		\label{eq: A otimes B-1}
	\end{align}
	holds for all selfadjoint operators $A, B \in B(\HH)$ such that %there are constants $m,M>0$ for which 
	$\Sp(A), \Sp(B) \subseteq [a,b] \subseteq (0,\infty)$ is determined by $k = (a^2+b^2)/(ab)$.	
	Here, $I$ denotes the identity on $\HH \otimes \HH$.
\end{lemma}
\begin{proof}
	First, note that the minimum constant $k$ for which the scalar inequality 
	\begin{align*}
		\frac{x}{y} + \frac{y}{x} \:\leqs\: k
	\end{align*}
	holds for all real numbers $x,y$ such that $x,y \in [a,b]$ is given by $k = (a/b)+(b/a)$.
	
	For selfadjoint operators $A$ and $B$ such that  
	$\Sp(A), \Sp(B) \subseteq [a,b] \subseteq (0,\infty)$, we have
	$\norm{A}, \norm{B} \in [a,b]$ and hence
	\begin{align*}
		\norm{A \otimes B^{-1} + A^{-1} \otimes B} 
		\:&\leqs\: \norm{A \otimes B^{-1}} + \norm{A^{-1} \otimes B} \\
		\:&=\: \norm{A} \norm{B}^{-1} + \norm{A}^{-1} \norm{B} \\
		\:&\leqs\: \dfrac{a^2+b^2}{ab}.
	\end{align*}
	Thus, we obtain the inequality \eqref{eq: A otimes B-1}. 
	The constant $(a^2+b^2)/(ab)$ cannot be improved since
	the case $A = aI_{\HH}$ and $B=bI_{\HH}$ is reduced to the scalar case.	
\end{proof}

\subsection{Kantorovich type integral inequalities}
%A related inequality to \eqref{eq: A otimes B-1} is given in \cite{Kijima}.
The following theorem is a Kantorovich type integral inequality.
%relative to the tensor product of operators. 

\begin{theorem} \label{thm: Kant integral W_t tensor}
	Under Main hypothesis, the following integral inequality holds
	\begin{equation} \label{eq: Kant ineq for tensor}
	%\begin{split}
		\int_{\Om} W_t^{\frac{1}{2}} A_t W_t^{\frac{1}{2}} \,d \mu(t) \otimes_s \IntAweight{A_t^{-1}} %\\
		\:\leqs\: \Kant \left( \int_{\Om} W_t \,d\mu(t) \right)^{ \otimes 2}.
	%\end{split}
	\end{equation}
	Moreover, the constant $(a^2+b^2)/(2ab)$ is best possible.
\end{theorem}
\begin{proof}
	For convenience, let us denote
	\begin{align*}
		X \:=\: \int_{\Om} {W_{t}}^{\frac{1}{2}} {A_t} {W_{t}}^{\frac{1}{2}} \,d \mu(t) \, \text{ and } \,
		Y \:=\: \int_{\Om} {W_{t}}^{\frac{1}{2}} {A_t}^{-1} {W_{t}}^{\frac{1}{2}} \,d \mu(t).
	\end{align*}
	It follows from the property \eqref{eq: int and tensor product} that
	\begin{align*}	
	%\int_{\Om} &{W_{t}}^{\frac{1}{2}} {A_t} {W_{t}}^{\frac{1}{2}} \,d \mu(t) 
	%\otimes \int_{\Om} {W_{t}}^{\frac{1}{2}} {A_t}^{-1} {W_{t}}^{\frac{1}{2}} \,d \mu(t)  \\
	X \otimes Y 
		\:&=\: \int_{\Om} \left( \IntWeightT{A_t} \right) \otimes W_r^{\frac{1}{2}} A_r^{-1} W_r^{\frac{1}{2}}\, d\mu(r) \\
		%&= \int_{\Om} \int_{\Om} \left( W_t^{\frac{1}{2}} A_t W_t^{\frac{1}{2}} \otimes W_r^{\frac{1}{2}} A_r^{-1} W_r^{\frac{1}{2}} 
		%	\right) \,d\mu(r)\,d\mu(t) \\
		\:&=\:  	\iint_{\Om^2} \left( W_t^{\frac{1}{2}} A_t W_t^{\frac{1}{2}} \otimes W_r^{\frac{1}{2}} A_r^{-1} W_r^{\frac{1}{2}} 
			\right) \,d\mu(r)\,d\mu(t).
	\end{align*}
Similarly, we have
	\begin{align*}
		%\int_{\Om} &{W_{t}}^{\frac{1}{2}} {A_t^{-1}} {W_{t}}^{\frac{1}{2}} \,d \mu(t) 
		%\otimes \int_{\Om} {W_{t}}^{\frac{1}{2}} A_t {W_{t}}^{\frac{1}{2}} \,d \mu(t)  \\
		Y \otimes X	&=  	\iint_{\Om^2} \left( W_t^{\frac{1}{2}} {A_t}^{-1} W_t^{\frac{1}{2}} 
			\otimes W_r^{\frac{1}{2}} A_r W_r^{\frac{1}{2}} 
			\right) \,d\mu(r)\,d\mu(t).
		\end{align*}
It follows that
\begin{align*}
2(X \otimes_s Y) \:&=\: 
		\iint_{\Om^2} \left(W_t^{\frac{1}{2}} A_t W_t^{\frac{1}{2}} 
			\otimes W_r^{\frac{1}{2}} A_r^{-1} W_r^{\frac{1}{2}} 
			+ W_t^{\frac{1}{2}} A_t^{-1} W_t^{\frac{1}{2}} \otimes W_r^{\frac{1}{2}} A_r W_r^{\frac{1}{2}} \right) \,d\mu(r) \,d\mu(t) \\
		\:&=\: \iint_{\Om^2} (W_t \otimes W_r)^{\frac{1}{2}} 
			\left(A_t \otimes A_r^{-1} + A_t^{-1} \otimes A_r \right)
			(W_t \otimes W_r)^{\frac{1}{2}} \,d\mu(r) \,d\mu(t).
	\end{align*}
	By making use of Lemma \ref{lem: Lemma 1} and the property \eqref{eq: int and tensor product}, 
	we obtain
	\begin{align*}
		%\int_{\Om} {W_{t}}^{\frac{1}{2}} &{A_t} {W_{t}}^{\frac{1}{2}} \,d \mu(t) 
		%\otimes \int_{\Om} {W_{t}}^{\frac{1}{2}} {A_t}^{-1} {W_{t}}^{\frac{1}{2}} \,d \mu(t)  \\
		X \otimes_s Y	&\:\leqs\: \dfrac{1}{2} \int_{\Om} \int_{\Om} \dfrac{a^2+b^2}{ab}(W_t \otimes W_r)  
				\,d\mu(r) \,d\mu(t) \\
			&\:=\:	\dfrac{a^2+b^2}{2ab} \int_{\Om} \left( \int_{\Om} W_r\, d\mu(r) \right) 
				\otimes W_t \, d\mu(t) \\
			&\:=\: \dfrac{a^2+b^2}{2ab} \int_{\Om} W_t \,d\mu(t) \otimes \int_{\Om} W_t \,d\mu(t).
	\end{align*}
	Therefore, we arrive at \eqref{eq: Kant ineq for tensor}.
	The best possibility of the constant $(a^2+b^2)/(2ab)$ also comes from Lemma \ref{lem: Lemma 1}.
\end{proof}

As a special case, we obtain a discrete version of the integral inequality \eqref{eq: Kant ineq for tensor}
as follows.

\begin{corollary}
	For each $i=1,2,\dots,n$, let $A_i \in \A$ be a selfadjoint operator such that 
	$\Sp(A_i) \subseteq [a,b] \subseteq (0,\infty)$
	and let $W_i$ be a positive operator in $\A$.
	Then we have
	\begin{align}
		\sum_{i=1}^n W_i^{\frac{1}{2}} A_i W_i^{\frac{1}{2}} \otimes_s \sum_{i=1}^n W_i^{\frac{1}{2}} A_i^{-1} W_i^{\frac{1}{2}}
		\;\leqs\; \Kant \left(\sum_{i=1}^n W_i \right)^{\otimes 2}.
	\end{align}
\end{corollary}
\begin{proof}
	Take $\Om=\{1,2,\dots,n\}$ and set $\mu$ to be the counting measure 
	in Theorem \ref{thm: Kant integral W_t tensor}.
\end{proof}

The next result is an integral inequality of Kantorovich type in which
the weights are scalars.

\begin{corollary} \label{cor: Kant integral w(t) tensor}
	%Let $(A_t)_{t \in \Om}$  be a bounded continuous field of strictly positive operators
%in $\A$ such that the function $t \mapsto \norm{A_t}$ is Lebesgue integrable on $\Omega$ 
%and $\Sp(A_t) \subseteq [a,b] \subseteq (0,\infty)$ for each $t \in \Om$.
	Assume Main hypothesis.	
	For any continuous function $w:\Om \to [0,\infty)$, we have
	%the following integral inequality holds
	\begin{equation} \label{eq: Kant ineq for tensor - weight}
	%\begin{split}
		\int_{\Om} w(t) A_t  \,d \mu(t) \otimes_s \int_{\Om} w(t) A_t^{-1}  \,d \mu(t) 
		\:\leqs\: \Kant \left( \int_{\Om} w(t) \,d\mu(t)\right)^2 I.
	%\end{split}
	\end{equation}
\end{corollary}
\begin{proof}
	Set $W_t =w(t)I$ for each $t \in \Om$ in Theorem \ref{thm: Kant integral W_t tensor}.
\end{proof}

The following result is a discrete version of the inequality \eqref{eq: Kant ineq for tensor - weight}.

\begin{corollary} \label{cor: reverse weight AM_HM}
	For each $i=1,2,\dots,n$, let $A_i \in \A$ be a selfadjoint operator such that 
	$\Sp(A_i) \subseteq [a,b] \subseteq (0,\infty)$
	and let $w_i \geqs 0$ be a constant.
	Then %we have
	\begin{align}
		\left(\sum_{i=1}^n w_i A_i \right)  \otimes_s \left( \sum_{i=1}^n w_i  A_i^{-1} \right) 
		\leqs \Kant \left(\sum_{i=1}^n w_i \right)^2 I.
	\end{align}
\end{corollary}
\begin{proof}
	Take $\Om=\{1,2,\dots,n\}$ and set $\mu$ to be the counting measure 
	in Corollary \ref{cor: Kant integral w(t) tensor}.
\end{proof}

From this corollary, when the weight $w_i$ is $1/n$ for each $i$, then
	\begin{equation} \label{eq: Kant as reverse AM-HM}
		\frac{1}{n} \left( A_1 + A_2 +\dots +A_n \right)  \otimes 
		\frac{1}{n} \left( A_1^{-1} + A_2^{-1} +\dots +A_n^{-1} \right) 
		\leqs \Kant I.
	\end{equation}
Recall that the harmonic mean of $A_1,A_2, \dots, A_n$ is given by
\begin{align*}
	n (A_1^{-1} + A_2^{-1} +\dots +A_n^{-1})^{-1}.
\end{align*}
%The inequality \eqref{eq: Kant as reverse AM-HM} can be viewed as a reverse arithmetic-harmonic mean (AM-HM)
%inequality for operators concerning tensor product.
Hence, Corollary \ref{cor: reverse weight AM_HM} provides a reverse weighted AM-HM inequality.
%Theorem \ref{thm: Kant integral W_t tensor} is therefore an integral analogue of 
%reverse weighted AM-HM inequality involving tensor product of operators.

\section{Kantorovich integral inequalities involving operator means}

In this section, we establish integral analogues of Kantorovich inequality involving operator means.
To begin with, recall some fundamental facts in Kubo-Ando theory of operator means \cite{Kubo-Ando}; 
see also 
%\cite[Chapter 4]{Bhatia_positive def matrices}, 
\cite[Section 3]{Hiai} and \cite[Chapter 5]{Hiai-Petz}.

\subsection{Preliminaries on operator means}
An \emph{(operator) connection} is a binary operation $\sm$ assigned to each pair of positive operators
such that for all $A,B,C,D \geqs 0$:
\begin{enumerate}
	\item[(M1)] (joint) monotonicity: $A \leqs C, B \leqs D \implies A \sm B \leqs C \sm D$
	\item[(M2)] transformer inequality: $C(A \sm B)C \leqs (CAC) \sm (CBC)$
	\item[(M3)] (joint) continuity from above:  for $A_n,B_n \in B(\HH)^+$,
                if $A_n \downarrow A$ and $B_n \downarrow B$,
                 then $A_n \sm B_n \downarrow A \sm B$.
                 Here, $X_n \downarrow X$ indicates that $(X_n)$ is a decreasing sequence
                 converging strongly to $X$.
\end{enumerate}
Using (M2), every operator connection $\sigma$ is invariant under congruence transformations 
in the sense that
\begin{align}
	C(A \sm B)C \:=\: (CAC) \sm (CBC),  \label{eq: congruent invariance}
\end{align}
for $A,B \geqs 0$ and $C>0$. Moreover, every connection $\sigma$ satisfies
%The positive homogeneity and the concavity of $\sigma$
%together imply the following property:
\begin{align}
	(A+B) \sm (C+D) \geqs (A \sm C) + (B \sm D),  \label{eq: mean - concave + pos hom}
\end{align}
for any $A,B,C,D \geqs 0$.

An \emph{operator mean} is a connection $\sigma$ with fixed point property $A \sm A =A$ 
for all $A \geqs 0$.
%For each $\ap \in [0,1]$, we define the $\ap$-weighted geometric mean 
%for invertible positive operators $A$ and $B$ as follows:
%\begin{align*}
%	A \,\#_{\ap}\, B \:=\: A^{\frac{1}{2}} (A^{-\frac{1}{2}} B A^{-\frac{1}{2}})^{\ap} A^{\frac{1}{2}}.
%\end{align*}

A major core of Kubo-Ando theory is the one-to-one correspondence between operator connections
and operator monotone functions.
Recall (e.g. \cite[Chapter 4]{Hiai-Petz}) that a continuous
function $f: [0,\infty) \to \R$ is said to be \emph{operator monotone} if
\begin{align*}
	A \leqs B \implies f(A) \leqs f(B) 
\end{align*}
holds for any positive operators $A$ and $B$. 
%See more information about operator monotone functions in 
%\cite[Chapter 4]{Hiai-Petz}.

\begin{proposition} \label{prop: Kubo-Ando}
(\cite[Theorem 3.4]{Kubo-Ando}) 
	Given an operator connection $\sigma$, there is a unique operator monotone function 
	$f:[0,\infty) \to [0,\infty)$
	such that
\begin{equation} 
	f(A) \:=\: I \,\sigma\, A,  \quad A \geqs 0.  \label{eq: f(A) = I sm A}
\end{equation}
In fact, the map $\sigma \mapsto f$ is a bijection. 
\end{proposition}
Such a function $f$ is called the \emph{representing function} of $\sigma$. 
%For example, the representing function of $\#_{\ap}$ is the operator monotone function $f(x)=x^{\ap}$.
	
%In order to prove the main result in this section, recall the following fact about operator connections.

\begin{lemma}[\cite{Arlinskii}] \label{lem: norm A sm B}
	Let $\sigma$ be an operator connection. Then for all positive operators $A$ and $B$ in $B(\HH)$,
	we have
	\begin{align*}
		\norm{A \,\sigma\, B} \leqs \norm{A} \,\sigma\, \norm{B}.
	\end{align*}
	Here, the connection $\sigma$ on the right hand side is the induced connection on $[0,\infty)$ defined
	by $(a \,\sigma\, b)I = aI \,\sigma\, bI$ for any $a,b \geqs 0$.
\end{lemma}

We say that a function $f: [0,\infty) \to \R$ is \emph{super-multiplicative}
if $f(xy) \geqs f(x)f(y)$ for all $x,y \geqs 0$.

\begin{lemma} \label{lem: f is super multiplicative}
	Let $\sigma$ be an operator connection associated with an operator monotone function 
	$f:[0,\infty) \to [0,\infty)$.
	If $f$ is super-multiplicative, then
	\begin{align}
		(A \,\sigma\, C) \otimes_s (B \,\sigma\,D) \leqs (A \otimes_s B) \,\sigma\,(C \otimes_s D)
		\label{eq: A sm C otimes B sm D}
	\end{align}
	for all positive operators $A,B,C,D$.
\end{lemma}
\begin{proof}
	By a continuity argument using $(M1)$ and $(M3)$, we may assume that $A$ and $B$ are strictly positive.
	Putting $X = A^{-\frac{1}{2}} C  A^{-\frac{1}{2}} $ and $Y = B^{-\frac{1}{2}} D  B^{-\frac{1}{2}} $
	yields
	\begin{align*}
		(A \,\sigma\, C) \otimes (B \,\sigma\,D)
		%&= A^{\frac{1}{2}} (I \,\sigma\, X)A^{\frac{1}{2}}
		%	\otimes B^{\frac{1}{2}} (I \,\sigma\, Y)B^{\frac{1}{2}} \\
		&= (A \otimes B)^{\frac{1}{2}} [(I \,\sigma\, X)\otimes(I \,\sigma\, Y)] (A \otimes B)^{\frac{1}{2}}\\
		&= (A \otimes B)^{\frac{1}{2}} [f(X) \otimes f(Y)] (A \otimes B)^{\frac{1}{2}}\\
		&\leqs (A \otimes B)^{\frac{1}{2}} [f(X \otimes Y)] (A \otimes B)^{\frac{1}{2}}\\
		&= (A \otimes B)^{\frac{1}{2}} [I \,\sigma\, (X \otimes Y)] (A \otimes B)^{\frac{1}{2}}\\
		&= (A \otimes B) \,\sigma\,(C \otimes D).
	\end{align*}
	Here, we use the congruent invariance 
	\eqref{eq: congruent invariance} and the property \eqref{eq: f(A) = I sm A}.
	Now, 
	\begin{align*}
		(A \,\sigma\, C) &\otimes (B \,\sigma\, D) + (B \,\sigma\, D) \otimes (A \,\sigma\, C) \\
			\,&\leqs\, (A \,\otimes\, B) \,\sigma\, (C \,\otimes\, D) 
				+ (B \,\otimes\, A) \,\sigma\, (D \,\otimes\, C) \\
			\,&\leqs\, \left[(A \,\otimes\, B) + (B \,\otimes\, A) \right] \,\sigma\,
				 \left[(C \,\otimes\, D) + (D \,\otimes\, C)\right]. 
	\end{align*}
	Hence, we obtain \eqref{eq: A sm C otimes B sm D}.
\end{proof}

\subsection{Kantorovich type integral inequalities
involving operator means}
The following result can be regarded as a Kantorovich type integral inequality
concerning operator means.

\begin{theorem} \label{thm: Kant - mean}
	Assume Main hypothesis. Let $(B_t)_{t \in \Omega}$ be a bounded continuous field of strictly positive operators such that $\Sp(B_t) \subseteq [a,b]$ for each $t \in \Omega$. 
	Let $\sigma$ be an operator mean with a super-multiplicative representing function. 
	Then %the following integral inequality holds:
		\begin{equation} \label{eq: Kant ineq tensor - mean}
	\begin{split}
		\int_{\Om} W_t^{\frac{1}{2}} (A_t \,\sigma\, B_t) W_t^{\frac{1}{2}} \,d \mu(t) 
		\:&\otimes_s\: 
		\IntAweight{(A_t^{-1} \,\sigma\, B_t^{-1})} \\
		\:&\leqs\: \Kant \left( \int_{\Om} W_t \,d\mu(t)  \right)^{\otimes 2}.
	\end{split}
	\end{equation}
\end{theorem}
\begin{proof}
	The function $t \mapsto W_t^{\frac{1}{2}} (A_t \,\sigma\, B_t) W_t^{\frac{1}{2}}$ is Bochner integrable
	due to the norm estimate in Lemma \ref{lem: norm A sm B}.
	It follows that
	\begin{align*}
		&\int_{\Om} W_t^{\frac{1}{2}} (A_t \sigma B_t) W_t^{\frac{1}{2}}  \,d\mu(t)
			\otimes_s \int_{\Om} W_t^{\frac{1}{2}} (A_t^{-1} \, 
			\sigma\, B_t^{-1}) W_t^{\frac{1}{2}} \,d\mu(t) \\
		&\leqs \int_{\Om} \left( W_t^{\frac{1}{2}} A_t W_t^{\frac{1}{2}} \,\sigma\, 
			W_t^{\frac{1}{2}} B_t W_t^{\frac{1}{2}} \right) \,d\mu(t) 
			\otimes_s \int_{\Om} \left( W_t^{\frac{1}{2}} A_t^{-1} W_t^{\frac{1}{2}} \,\sigma\, 
			W_t^{\frac{1}{2}} B_t^{-1} W_t^{\frac{1}{2}} \right) \,d\mu(t) \\
		&\qquad \text{(since } \sigma \text{ satisfies the transformer inequality (M2))} \\
		&\leqs \left[\IntWeightT{A_t} \,\sigma\, \IntWeightT{B_t} \right] \\
			&\qquad \otimes_s \left[\IntWeightT{A_t^{-1}} \,\sigma\, \IntWeightT{B_t^{-1}} \right] \\
			&\qquad \text{(since } \sigma \text{ satisfies the property 
			\eqref{eq: mean - concave + pos hom})} \\
		&\leqs \left[\IntWeightT{A_t} \otimes_s \IntWeightT{A_t^{-1}} \right] \\
			&\qquad \sigma\, \left[\IntWeightT{B_t} \otimes_s \IntWeightT{B_t^{-1}} \right] 
			\qquad \text{(by Lemma \ref{lem: f is super multiplicative})} \\
		&\leqs  \Kant \left( \int_{\Om} W_t \,d\mu(t)\right)^{\otimes 2}  
			\;\sigma\;  \Kant \left( \int_{\Om} W_t \,d\mu(t)\right)^{\otimes 2}   %\\
			\qquad \text{(by Theorem \ref{eq: Kant ineq for tensor})} \\
		&= \Kant \left(\int_{\Om} W_t \,d\mu(t) \right)^{\otimes 2}
			\quad \text{(since } \sigma \text{ satisfies the fixed point property)}.
	\end{align*}
	The proof is complete.
\end{proof}

Theorem \ref{thm: Kant - mean} can be reduced to Theorem \ref{thm: Kant integral W_t tensor}
by setting $A_t = B_t$ for all $t \in \Om$.
The next result is discrete version of the inequality \eqref{eq: Kant ineq tensor - mean}.

\begin{corollary}
	For each $i=1,2,\dots,n$, let $A_i, B_i \in \A$ be selfadjoint operators such that 
	$\Sp(A_i), \Sp(B_i) \subseteq [a,b] \subseteq (0,\infty)$
	and let $W_i$ be a positive operator in $\A$.
	Then we have
	\begin{align}
		\sum_{i=1}^n W_i^{\frac{1}{2}} (A_i \sm B_i) W_i^{\frac{1}{2}} \otimes_s 
		\sum_{i=1}^n W_i^{\frac{1}{2}} (A_i^{-1} \sm B_i^{-1}) W_i^{\frac{1}{2}}
		\leqs \Kant \left(\sum_{i=1}^n W_i \right)^{\otimes 2}.
	\end{align}
\end{corollary}
\begin{proof}
	Take $\Om=\{1,2,\dots,n\}$ and set $\mu$ to be the counting measure 
	in Theorem \ref{thm: Kant - mean}.
\end{proof}

%The next result is also a special case of Theorem \ref{thm: Kant - mean} in which an operator mean
%is specified.

%\begin{corollary} \label{cor: Kant - geom mean}
%	Assume Main hypothesis. 
%	Let $(B_t)_{t \in \Omega}$ be a bounded continuous field of strictly positive operators such that $\Sp(B_t) 
%\subseteq [a,b]$ for each $t \in \Omega$. 
%	Then the following integral inequality holds for any $\ap \in [0,1]$:
%		\begin{equation} \label{eq: Kant ineq tensor - geom mean}
%	\begin{split}
%		\int_{\Om} W_t^{\frac{1}{2}} (A_t \,\#_{\ap}\, B_t) W_t^{\frac{1}{2}} \,d \mu(t) &\otimes_s 
%		\IntAweight{(A_t^{-1} \,\#_{\ap}\, B_t^{-1})} \\
%		&\leqs \Kant \left( \int_{\Om} W_t \,d\mu(t) \right)^{\otimes 2}.
%	\end{split}
%	\end{equation}
%\end{corollary}
%\begin{proof}
%	From Theorem \ref{thm: Kant - mean}, set $\sigma$ to be the 
%	$\ap$-weighted geometric mean $\#_{\ap}$.
%	Note that its representing function $f(x)=x^{\ap}$ is super-multiplicative.
%\end{proof}

\section{Further operator integral inequalities}

Theorem \ref{eq: Kant ineq for tensor} can be extended
in the following way: 

\begin{theorem} \label{thm: Kant ineq tensor f(x) f(x-1)}
	Assume Main hypothesis. Let $f$ be a continuous real-valued function defined on 
	$[a,b] \cup [1/b, 1/a]$ such that $f(x)f(1/x) \leqs 1$ for all $x \in [a,b]$.
	Suppose that $f([a,b]) \subseteq [a,b]$ or $f([a,b]) \subseteq [1/b, 1/a]$. 
	Then 
	%Then %the following integral inequality holds:
	\begin{equation} \label{eq: Kant ineq for tensor - f}
	\begin{split}
		\int_{\Om} W_t^{\frac{1}{2}} f(A_t) W_t^{\frac{1}{2}} \,d \mu(t) \:&\otimes_s\: \IntAweight{f(A_t^{-1})} \\
		&\leqs \Kant \left( \int_{\Om} W_t \,d\mu(t) \right)^{\otimes 2}.
	\end{split}
	\end{equation}
\end{theorem}
\begin{proof}
	Since $\Sp(A_t^{-1}) \subseteq [1/b, 1/a]$ for each $t$, 
	the function $t \mapsto W_t^{\frac{1}{2}} f(A_t^{-1}) W_t^{\frac{1}{2}}$ 
	is Bochner integrable by Proposition \ref{prop: Bochner integralability}.
	The assumption also implies that $f(A_t^{-1}) \leqs f(A_t)^{-1}$ for each $t \in \Om$.
	The desired result now follows from Theorem \ref{thm: Kant integral W_t tensor}.
	Note that the constant $(a^2+b^2)/(2ab)$ is not affected.
\end{proof}

Theorem \ref{thm: Kant ineq tensor f(x) f(x-1)} % where $f:[a,b] \to [a,b]$ 
is reduced to Theorem \ref{thm: Kant integral W_t tensor} by setting $f(x)=x$ or $f(x)=1/x$.

%\begin{corollary} \label{cor: Gruss ineq}
	%Let $(A_t)_{t \in \Om}$ be a bounded continuous field of strictly positive operators
%in $\A$ such that the map $t \mapsto \norm{A_t}$ is Lebesgue integrable on $\Omega$ 
%and	$\Sp(A_t) \subseteq [a,b] \subseteq (0,\infty)$ for each $t \in \Om$.
%Assume Main hypothesis.
%Then, for any continuous function $f:[a,b] \to [0,\infty)$, we have
%	\begin{equation} 
%	\begin{split}
%		\int_{\Om} A_t f(A_t) \,d \mu(t) \otimes_s \int_{\Om} A_t^{-1} f(A_t) \,d \mu(t)  
%		\:\leqs\: \Kant \left( \int_{\Om} f(A_t) \,d\mu(t) \right)^{ \otimes 2}.
%	\end{split}
%	\end{equation}
%\end{corollary}
%\begin{proof}
%	Put $W_t =f(A_t)$ for each $t \in \Omega$ in Theorem \ref{thm: Kant integral W_t tensor}.
%	Then $(W_t)_{t \in \Om}$ is a continuous field of positive operators.
%\end{proof}

\begin{corollary} \label{cor: further ineq}
	Assume the hypothesis of Theorem \ref{thm: Kant ineq tensor f(x) f(x-1)}.  
	For any continuous function $g:[a,b] \to [0,\infty)$, we have 
	%Then %the following integral inequality holds:
	\begin{equation}  
	%\begin{split}
		\int_{\Om}  f(A_t) g(A_t) \,d \mu(t) \:\otimes_s\: \int_{\Om}  f(A_t^{-1}) g(A_t) \,d \mu(t) 
		\:\leqs\: \Kant \left( \int_{\Om} g(A_t) \,d\mu(t) \right)^{\otimes 2}.
	%\end{split}
	\end{equation}
\end{corollary}
\begin{proof}
	Set $W_t = g(A_t)$ for each $t \in \Omega$ in Theorem \ref{thm: Kant ineq tensor f(x) f(x-1)}.
	Then $(W_t)_{t \in \Omega}$ is a continuous field of positive operators.
\end{proof}

The next result can be viewed as a generalization of Gr\"{u}ss inequality.

\begin{corollary} \label{cor: Gruss ineq -2}
Assume Main hypothesis.
For any $\ld \in \R$, we have
	\begin{equation} 
	\begin{split}
		\int_{\Om} A_t^{\ld+1}  \,d \mu(t) \otimes_s \int_{\Om} A_t^{\ld-1}  \,d \mu(t) 
		\:\leqs\: \Kant \left( \int_{\Om} A_t^{\ld} \,d\mu(t) \right)^{ \otimes 2}.
	\end{split}
	\end{equation}
\end{corollary}
\begin{proof}
	Put $f(x)=x$ and $g(x)= x^{\ld}$ in Corollary \ref{cor: further ineq}.
\end{proof}

The case $\ld=1$ and $\mu(\Omega)=1$ in this corollary is a Gr\"{u}ss type integral inequality for tensor product of operators:

	\begin{equation} \label{eq: Gruss ineq for tensor}
	\begin{split}
		\int_{\Om} A_t^{2}  \,d \mu(t) \otimes_s I 
		\:\leqs\: \Kant \left( \int_{\Om} A_t  \,d\mu(t) \right)^{ \otimes 2}.
	\end{split}
	\end{equation}

\begin{theorem} \label{cor: Kant - mean}
	Assume Main hypothesis. 
	Suppose that $1 \in [a,b]$.
	For any super-multiplicative operator monotone function $f:[0,\infty) \to [0,\infty)$ such that $f(1)=1$, 
	we have
	%Then %the following integral inequality holds:
		\begin{equation} %\label{eq: Kant ineq tensor - mean}
	\begin{split}
		\int_{\Om} W_t^{\frac{1}{2}} f(A_t) W_t^{\frac{1}{2}} \,d \mu(t) 
		\:&\otimes_s\: 
		\int_{\Om} W_t^{\frac{1}{2}} f(A_t^{-1}) W_t^{\frac{1}{2}} \,d \mu(t) \\
		\:&\leqs\: \Kant \left( \int_{\Om} W_t \,d\mu(t)  \right)^{\otimes 2}.
	\end{split}
	\end{equation}
\end{theorem}
\begin{proof}
	By Proposition \ref{prop: Kubo-Ando}, there is an operator mean $\sigma$ such that $f(A)= I \,\sigma\, A$
	for any $A \geqs 0$. The desired result now follows from Theorem \ref{thm: Kant - mean} by considering $I \,\sigma A_t$ instead of $A_t \,\sigma\, B_t$.
\end{proof}

%For each $\ap \in [0,1]$, recall that the function $f(x)=x^{\ap}$ is operator monotone. 
%Note that this function is super-multiplicative and
%satisfies $f(1)=1$. Hence, under the hypothesis of Theorem \ref{cor: Kant - mean}, the following inequality
%holds for any $\ap \in [-1,1]$:
%	\begin{equation*} %\label{eq: Kant ineq tensor - geom mean 2}
%		\int_{\Om} W_t^{\frac{1}{2}} A_t^{\ap}  W_t^{\frac{1}{2}} \,d \mu(t) \otimes_s 
%		\IntAweight{A_t^{-\ap}} %\\
%		\:\leqs\: \Kant \left( \int_{\Om} W_t \,d\mu(t) \right)^{\otimes 2}.
%	\end{equation*}

\begin{corollary} \label{cor: last}
	Assume Main hypothesis. Suppose that $1 \in [a,b]$.
	%Then the following integral inequality holds for any $\ap \in [-1,1]$:
	For any $\ap \in [-1,1]$ and a continuous function $g:[a,b] \to [0,\infty)$, we have
		\begin{equation}  \label{eq: prelast} 
	\begin{split}
		\int_{\Om}  A_t^{\ap}  g(A_t) \,d \mu(t) \otimes_s 
		\int_{\Om}  A_t^{-\ap}  g(A_t) \,d \mu(t)
		\:\leqs\: \Kant \left( \int_{\Om} g(A_t) \,d\mu(t) \right)^{\otimes 2}.
	\end{split}
	\end{equation}
\end{corollary}
\begin{proof}
	Consider the operator monotone function $f(x)= x ^{\ap}$. 
Note that this function is super-multiplicative and
satisfies $f(1)=1$.
The desired result now follows by replacing $W_t$ by $g(A_t)$ in Theorem \ref{cor: Kant - mean}.
	%Since the operation $\otimes_s$ is commutative, we may assume that	$\ap \in[0,1]$.
	%From Theorem \ref{thm: Kant - mean}, set $\sigma$ to be the 
	%$\ap$-weighted geometric mean $\#_{\ap}$.
	%Note that its representing function is given by $f(x)=x^{\ap}$, which is super-multiplicative.
	%From \eqref{eq: Kant ineq tensor - geom mean}, 
	%replacing $A_t$ and $B_t$ with $I$ and $A_t$, respectively yields the inequality
	%\eqref{eq: Kant ineq tensor - geom mean 2} via the property \eqref{eq: f(A) = I sm A}.
	%For the case $\ap \in [-1,0]$, note that if $\Sp(A_t) \subseteq [a,b]$, then
	%$\Sp(A_t^{-1}) \subseteq [1/b,1/a]$ for each $t \in \Om$. 
	%Replacing $A_t$ with $A_t^{-1}$ for each $t \in \Om$ in 
	%\eqref{eq: Kant ineq tensor - geom mean} does not effect the constant $(a^2+b^2)/(2ab)$.
	%yields the desired result; here note also that
	%\begin{equation*}
	%	\frac{(1/b)^2+(1/a)^2}{2(1/b)(1/a)} 
	%	\:=\: \Kant.
	%\end{equation*}
\end{proof}

Under the hypothesis of Corollary \ref{cor: last}, we have an interesting operator inequality.
For each $\lambda \in \R$, putting $g(x)= x^{\lambda}$ in \eqref{eq: prelast} yields
	\begin{equation}  
	\begin{split}
		\int_{\Om}  A_t^{\lambda+\ap} \,d \mu(t) \otimes_s 
		\int_{\Om}  A_t^{\lambda-\ap}  \,d \mu(t)
		\:\leqs\: \Kant \left( \int_{\Om} A_t^{\lambda} \,d\mu(t) \right)^{\otimes 2}.
	\end{split}
	\end{equation}
	
Discrete versions for the inequalities in this section can be obtained by considering $\Omega$ to be a finite space equipped with the counting measure.
\\

{\bf Acknowledgement.}
  This research was supported by King Mongkut's Institute
of Technology Ladkrabang Research Fund grant no. KREF045710.

\end{document}